\documentclass[12pt,reqno]{amsart}

\advance \topmargin by -\headheight
\advance \topmargin by -\headsep
\evensidemargin \oddsidemargin
\usepackage{amsmath}
\usepackage{amsfonts}
\usepackage{stmaryrd}
\usepackage{amssymb}
\usepackage{amsthm}
\usepackage{csquotes}

\usepackage{mathrsfs}
\usepackage{dsfont}
\usepackage{bm}

\numberwithin{equation}{section}

\newtheorem{theorem}{Theorem}[section]

\newtheorem{lemma}[theorem]{Lemma}
\newtheorem{Proposition}[theorem]{Proposition}
\newtheorem{Conjecture}[theorem]{Conjecture}

\title[Difference sets in higher dimensions]{Difference sets in higher dimensions}
\author[Akshat Mudgal]{Akshat Mudgal}
\date{} 
\address{Department of Mathematics, Purdue University, 150 N. University Street, West Lafayette, IN 47907-2067, USA }
\address{School of Mathematics, University of Bristol, Fry Building, Woodland Road, Bristol, BS8 1UG, UK}
\email{am16393@bristol.ac.uk, amudgal@purdue.edu}

\begin{document}

\maketitle
\begin{abstract}
Let $d \geq 3$ be a natural number. We show that for all finite, non-empty sets $A \subseteq \mathbb{R}^d$ that are not contained in a translate of a hyperplane, we have
\[ |A-A| \geq (2d-2)|A| - O_d(|A|^{1- \delta}),\]
where $\delta >0$ is an absolute constant only depending on $d$. This improves upon an earlier result of Freiman, Heppes and Uhrin, and makes progress towards a conjecture of Stanchescu. 
\end{abstract}

\section{Introduction}
Let $A, B$ be finite subsets of $\mathbb{R}^d$, for some $d \in \mathbb{N}$. We define the sumset 
\[ A + B = \{ a + b \ | \ a \in A, b \in B \}, \]
and the difference set
\[ A- B = \{ a-b \ | \ a \in A, b \in B \}. \]
The problem of finding sharp estimates for $|A-A|$, in terms of $|A|$, has been studied by various authors (see \cite{FHU1989}, \cite{St1998}, \cite{St2001}). Moreover, such estimates have been applied to improve results in geometry of numbers, including the classical theorem of Minkowski-Blichfeld (see \cite{GH2016}, \cite{Uh1980}, \cite{Uh1981}). Thus, in this paper, we study lower bounds for $|A-A|$, when $A$ is a finite subset of $\mathbb{R}^d$. 
\par

We begin by defining the notion of \emph{dimension} for a set. Given a non-empty set $A \subseteq \mathbb{R}^d$, we define $\dim(A)$ to be the dimension of the affine subspace spanned by $A$. If $\dim(A) = k$, we write $A$ to be a $k$-dimensional set or equivalently, say that dimension of $A$ is $k$. With this notation in hand, we state our main result. 

\begin{theorem} \label{diff}
Let $d \geq 3$ be a natural number and $A \subseteq \mathbb{R}^d$ be a finite, non-empty set such that $\dim(A) = d$. Then we have
\[ |A-A| \geq (2d-2)|A| - O_d(|A|^{1-\delta}),\]
where $\delta >0$ is an absolute constant only depending on $d$. 
\end{theorem}

When $d \geq 4$, the previously best known estimate on this problem was a result of Freiman, Heppes and Uhrin \cite{FHU1989}, who showed that
\[ |A-A| \geq (d+1)|A| - d(d+1)/2, \] 
for all non-empty, finite subsets $A$ of $\mathbb{R}^d$ such that $\dim(A) = d$. Thus for large values of $d$, we improve the previously best known lower bound almost by a factor of $2$. 
\par

As for the corresponding best known examples, when $d \geq 3$, Stanchescu \cite{St2001} constructed a sequence of $d$-dimensional sets $A_i$ such that $|A_i| \to \infty$ as $i \to \infty$ and
\begin{equation} \label{stanex}
 |A_i - A_i| = \bigg(2d - 2 + \frac{1}{d-1} \bigg)|A_i| - O_d(1), \ \text{for all} \ i \in \mathbb{N}.
 \end{equation} 
We provide a brief outline of this construction. We use $\{e_1, \dots, e_d\}$ to denote the standard basis for $\mathbb{Z}^d$, and we set $e_0 = (0, \dots, 0) \in \mathbb{Z}^d$. Furthermore, we define the set $T$ as 
\[ T = \{ e_0, e_1, \dots, e_{d-2} \}. \]
Next, for each natural number $i$, we write
\[ a_i = e_d - i e_{d-1}, \ \text{and} \ p_{i} = \{e_0, e_{d-1}, 2e_{d-1}, \dots, (i-1) e_{d-1}\}. \]
With this notation in hand, we let
\[ A_i = (T \cup (a_i - T)) + p_i . \]
We note that $A_i$ can be written as a union of $2d-2$ translates of the set $p_i$. Moreover, since $\dim(T) = d-2$, we see that $A_i$ is contained in two parallel hyperplanes. 
\par

Stanchescu \cite[Conjecture 1.2]{St2001} conjectured this example to be optimal. We state this conjecture below. 

\begin{Conjecture}
Let $d \geq 4$ be a natural number and $A \subseteq \mathbb{R}^d$ be a finite, non-empty set such that $\dim(A) = d$. Then we have
\[|A-A| \geq \bigg(2d - 2 + \frac{1}{d-1} \bigg)|A| - O_d(1). \] 
\end{Conjecture}

Thus, when $d \geq 4$, Theorem $\ref{diff}$ is within a factor of $1/(d-1)$ of the conjectured lower bound. 

\par

We remark that for smaller values of $d$, that is, when $d \in \{1,2,3 \}$, stronger results are known about lower bounds for $|A-A|$. In particular, if $A \subseteq \mathbb{R}$, we know that 
\[ |A-A| \geq 2|A|- 1, \]
and when $A \subseteq \mathbb{R}^2$ and $\dim(A) = 2$, work of Freiman, Heppes and Uhrin \cite{FHU1989} implies that
\[ |A-A| \geq 3|A|- 3. \]
When $d=3$, Stanchescu \cite{St1998} showed that 
\[ |A-A| \geq 4.5|A| - 9, \] 
for all non-empty, finite $A \subseteq \mathbb{R}^3$ such that $\dim(A) = 3$. Moreover, these inequalities can not be further strengthened. 
\par

We note that multiple variants of this problem have been studied in higher dimensions as well. In particular, one may consider lower bounds for cardinalities of sets of the from $A + \mathscr{L}(A)$, where $A \subseteq \mathbb{R}^d$ with $d \geq 2$, and $\mathscr{L}$ is some invertible linear transformation. The case when $\mathscr{L}$ is a particular type of integer dilation was studied by Balog and Shakan in \cite{BS2015}, which was then sharpened and generalised by the author in \cite{Ak2019}. Furthermore, the case of $\mathscr{L}$ being a rotation matrix in $\mathbb{R}^2$ was considered in \cite{Ak2019}, and we refer the reader to this paper for a more detailed exposition on this problem. 
\par

We now outline the structure of our paper. We dedicate \S2 to present some preliminary results that we will use in our paper. In \S3  we will reduce the proof of Theorem $\ref{diff}$ to considering the cases when $A$ can be covered by a small collection of parallel lines. We use \S4 to work on Proposition $\ref{lines}$ which handles the aforementioned reduced cases. 

\par

{\bf Acknowledgements.} The author's work is supported by a studentship sponsored by a European Research Council Advanced Grant under the European Union's Horizon 2020 research and innovation programme via grant agreement No.~695223. The author would like to thank Trevor Wooley for his guidance and direction, and the anonymous referee for many helpful comments.


\section{Preliminaries}

In our proof of Theorem $\ref{diff}$, we will use a standard inequality to move from difference sets to sumsets. We mention this result as stated in \cite[Corollary 2.12]{TV2006}.

\begin{lemma} \label{ruzcal}
Suppose that $U,V$ are two finite, non-empty sets in some abelian group $G$. Then 
\[ |U+V| \leq \frac{|U-V|^3}{|U||V|}.  \]
\end{lemma}

With sumset estimates in hand, we will then use a structure theorem to cover our set with a small collection of parallel lines. We state this result as shown in \cite[Theorem 1.5]{Ak2019}. 

\begin{lemma} \label{fri}
Let $A$ be a finite subset of $\mathbb{R}^d$ with $|A| = n$ where $n$ is large enough. If
\[ |A+A| \leq cn, \]
for some $c > 0$, then there exist parallel lines $l_1, l_2, \dots, l_r$ in $\mathbb{R}^d$, and constants $0 < \sigma \leq 1/2$ and $C_1 > 0$ depending only on $c$ such that 
\[ |A \cap l_1| \geq \dots \geq  |A \cap l_r| \geq  |A \cap l_1|^{1/2} \geq C_1^{-1} n^{\sigma}. \]
and 
\[ |A\setminus (l_1 \cup l_2 \cup \dots \cup l_r)| < C_1 c n^{1-\sigma}. \] 
\end{lemma}

We remark that while we will use Lemma $\ref{fri}$ to cover $A$ with translates of a one dimensional subspace, our strategy will be to find an affine hyperplane which has a large intersection with $A$ and consequently, cover $A$ with translates of this hyperplane. The core of the proof lies in analysing the interactions between extremal translates of this hyperplane. We will proceed with this in \S4. Another key ingredient in our proof will be a result of Ruzsa \cite[Corollary 1.1]{Ru1994} on sumsets of $d$-dimensional sets. 

\begin{lemma} \label{ruzsa}  
Let $A, B$ be finite, non-empty subsets of $\mathbb{R}^d$ such that $|A| \geq |B|$ and $\dim(A+B) = d$. Then we have
\[ |A+B| \geq |A| + d|B| - d(d+1)/2.\]
\end{lemma}

Furthermore, in some instances, we will use a standard inequality for sumsets, that is, for any finite, non-empty $A, B \subseteq \mathbb{R}^d$, we have
\begin{equation} \label{triv}
 |A+B| \geq |A| + |B| - 1. 
 \end{equation}
 
We use the rest of this section to prove a preliminary lemma. Our set up is as follows. Let $d \geq 4$ and $m \geq 2$ be natural numbers. Let $\mathcal{R}$ be a subspace of $\mathbb{R}^d$ such that $1 \leq \dim\mathcal{R} \leq d-1$. Let $Z$ be a finite, non-empty subset of $\mathbb{R}^{d}$ such that $Z$ is contained in $m$ cosets of $\mathcal{R}$. Thus, let $U = \{ z_1, \dots, z_m\}$ satisfy
\[ Z \subseteq U + \mathcal{R}. \]
After a suitable linear transformation, we can consider $U$ as a subset of $\mathbb{R}^{d'}$, where $d' \geq 1$ is some natural number such that $\mathbb{R}^{d'}$ is isomorphic to $\mathbb{R}^d / \mathcal{R}$. Moreover, we use the natural number $u$ to denote $\dim(U)$. Upon another suitable affine linear transformation, we can assume that $U \subseteq \mathbb{R}^u$. We will use $\psi : \mathbb{R}^d \to \mathbb{R}^u$ to denote the composition of the projection map from $\mathbb{R}^d$ to $\mathbb{R}^d/\mathcal{R}$, and the aforementioned affine linear transformations. With this notation in hand, we state our preliminary lemma.

\begin{lemma} \label{highplan}
If $u +1 \leq m \leq 2u$, then
\[ |(Z-Z) \setminus \mathcal{R}| \geq 2(2u + 1- m)|Z| - 16(u^2 + m^2). \]
\end{lemma}

This lemma roughly states that high dimensional difference sets can not lie too much in a lower dimensional subspace. This will perform a key role in our proof of Theorem $\ref{diff}$. Moreover, the methods we introduce to prove Lemma $\ref{highplan}$ are similar to the ideas we use to tackle Theorem $\ref{diff}$. 
\par

For ease of exposition, given $X \subseteq \mathbb{R}^d$, we write
\[ \psi(X) = \{ \psi(x) \ | \ x \in X \}. \]
Similarly, for any $Y \subseteq \mathbb{R}^{u}$, we denote $Y^{\psi} \subseteq \mathbb{R}^d$ to be
\[ Y^{\psi} = \{ x \in Z \cap \mathbb{R}^d \ | \ \psi(x) \in Y\}. \]
Thus, we see that
\[ \psi(Z) = U \ \text{and} \ U^{\psi} = Z. \]
Hence, it suffices to show that
\[  |(U^{\psi}-U^{\psi})\setminus \mathcal{R}| \geq 2(2u + 1- m)|U^{\psi}| - 16(u^2 + m^2). \]
\par

\begin{proof}[Proof of Lemma $\ref{highplan}$]
We prove our lemma by induction. Let $Q(m,u)$ be the statement of Lemma $\ref{highplan}$ for $m,u \in \mathbb{N}$ satisfying $u+1 \leq m \leq 2u$. Our base case is when $m = u+1$, that is, we begin with showing that $Q(u+1,u)$ holds. For our inductive step, we assume that $u+2 \leq m \leq 2u$. We then show that $Q(m,u)$ holds whenever $Q(m',u)$ and $Q(m'',u-1)$ hold for all $u+1 \leq m' < m$ and $u \leq m'' < \min(m,2u-1)$. We thus prove $Q(m,u)$ holds for all $m,u \in \mathbb{N}$ satisfying $u+1 \leq m \leq 2u$. 
\par

We now consider our base case, that is, when $U$ consists of $u+1$ affinely independent points. In this case,  we see that if $z_i, z_j, z_k, z_l \in U$ satisfy
\[ z_i - z_j = z_k - z_l,\]
then we have either $i = j$ or $i=k$. Thus $z_i - z_j$ is distinct for all $i \neq j$. This implies that the sets $z_i^{\psi} - z_j^{\psi}$ and $z_k^{\psi} - z_l^{\psi}$ are pairwise disjoint whenever $i \neq j,k$. Moreover, as each $z_i - z_j$ is non-zero, the set $z_i^{\psi} - z_j^{\psi}$ does not intersect $\mathcal{R}$. Hence, we have
\begin{align*}
 |(U^{\psi} - U^{\psi})\setminus \mathcal{R} | 
 & \geq \sum_{i \neq j} |z_i^{\psi} - z_j^{\psi}| \geq \sum_{i \neq j} (|z_i^{\psi}| + |z_j^{\psi}|- 1) \\
 & \geq 2u \sum_{i =1}^{u+1} |z_i^{\psi}| - (u+1)^2  = 2u |U^{\psi}| - (u+1)^2. 
\end{align*} 
Thus we have resolved our base case. We now move to the inductive step. 
\par

Let $C_U$ be the convex hull of $U$, and $F_1$ be the $(u-1)$-dimensional affine subspace of $\mathbb{R}^{u}$ that contains a $(u-1)$-dimensional facet of $C_U$. We use $l''$ to denote the line in $\mathbb{R}^u$ that is orthogonal to $F_1$, and then we cover $U$ with translates of $F_1$. We use $\| .\|_u$ to denote the Euclidean norm in $\mathbb{R}^u$. We write $F_2$ to be the translate of $F_1$ such that $F_2 \cap U \neq \emptyset$ and ${\| (l'' \cap F_2) - (l'' \cap F_1) \|}_{u}$ is maximised. Lastly, we write $U_i = F_i \cap U$, and $V_i = U \setminus U_i$ for all $1 \leq i \leq 2$. 
\par

As $U_1$ is $(u-1)$-dimensional, it contains at least $u$ affinely independent points $w_1, \dots, w_{u}$. We write $W = \{w_1,\dots, w_{u}\}$. We begin by claiming that 
\begin{equation} \label{cor7}
 |U_1^{\psi} - U_2^{\psi}|   =  |U_2^{\psi} - U_1^{\psi}| \geq  |y^{\psi}  - W^{\psi}|  + (2u+1-m)(|U_2^{\psi}| - |y^{\psi}|),
\end{equation}
where $y$ is some element of $U_2$. If $|U_2| = 1$, then this is trivial. Thus we can assume that $|U_2| \geq 2$. For a fixed $z \in U_2$, we write $U_{2,z} = U_2  \setminus \{z\}$. We consider the differences
\begin{equation} \label{wut}
 z - w_1, z- w_2, \dots, z- w_u. 
 \end{equation}
We see that our choice of $F_2$ implies that $U_1 - U_2$, $U_2 - U_1$ and $V_2 - V_2$ are pairwise disjoint. Consequently, as the differences in $\eqref{wut}$ lie in $U_2- U_1$, if any of them coincide with differences from $U-U$, they must be from $U_2 - U_1$. Moreover, if $z, z'$ are distinct elements of $U_2$ and $w_i,w_j,w_k,w_l$ lie in $W$ such that
\[ z- w_i = z' - w_j \ \text{and} \ z - w_k = z' - w_l, \]
then we have
\[ w_i - w_j = w_k - w_l. \]
Combining this with the fact that $w_1, \dots, w_{u}$ are affinely independent, we get
\[ w_i = w_k \ \text{and} \ w_j = w_l.\]
Thus 
\[| \{ z - w_1, \dots, z- w_u \} \cap (U_{2,z} - W)| \leq |U_{2,z}| \leq  m - 1 - u.\]
This implies that there are at least $u - (m-1-u) = 2u - (m-1)$ differences in $\eqref{wut}$, say $z - w_{1}, \dots, z - w_{2u+1-m}$ that are not contained in $(U_{2,z}-W)$. Thus, we have
 \begin{align*}
  |W^{\psi} - U_2^{\psi}|  =  |U_2^{\psi} - W^{\psi}| 
 & \geq  |U_{2,z}^{\psi}  - W^{\psi}| +  \sum_{i=1}^{2u+1 - m} |z^{\psi} - w_{i}^{\psi}|   \nonumber \\
 & \geq  |U_{2,z}^{\psi}  - W^{\psi}|  + (2u+1-m)|z^{\psi}|.
 \end{align*}
We iterate this argument for $U_{2,z}$ in place of $U_2$ till we get
\[ |U_1^{\psi} - U_2^{\psi}|   =  |U_2^{\psi} - U_1^{\psi}| \geq  |y^{\psi}  - W^{\psi}|  + (2u+1-m)(|U_2^{\psi}| - |y^{\psi}|), \]
where $y$ is some element of $U_2$. Thus we have proven $\eqref{cor7}$. 
\par
 
Let $U_1 = \{ z_1, z_2, \dots, z_q \}$, where $u \leq q \leq m-1$. We see that 
\[ y - z_1, y - z_2, \dots, y - z_q\]
are pairwise distinct. Hence, the sets
\[y^{\psi} - z_1^{\psi}, y^{\psi} - z_2^{\psi}, \dots, y^{\psi}- z_q^{\psi} \]
are pairwise disjoint. We use $\eqref{triv}$ to see that
\[ |y^{\psi} - z_i^{\psi}| \geq |y^{\psi}| + |z_i^{\psi}| - 1, \]
for each $1 \leq i \leq m-1$. Consequently, we have
\begin{align} \label{fj2}
 |y^{\psi} - U_1^{\psi}| & = \sum_{i=1}^{q}|y^{\psi} - z_i^{\psi}|   \geq q|y^{\psi}| + \sum_{i=1}^{q}|z_i^{\psi}| - q  = q|y^{\psi}| + |U_1^{\psi}| - q.
 \end{align}
 \par
Similarly, we have
\[ |y^{\psi}-W^{\psi}| = \sum_{i=1}^{u} |y^{\psi} - w_i^{\psi}| \geq u |y^{\psi}| \geq (2u + 1 - m)|y^{\psi}|. \]
This combines with $\eqref{cor7}$ to give
\begin{equation} \label{wut7}
 |U_1^{\psi} - U_2^{\psi}|   =  |U_2^{\psi} - U_1^{\psi}| \geq (2u+1-m) |U_2^{\psi}|.
 \end{equation}
 \par
 
As before, we note that $U_1 - U_2$, $U_2 - U_1$ and $V_2 - V_2$ are pairwise disjoint, and each difference in $U_1 - U_2$ is non-zero. Thus, we get the following decomposition 
\begin{align} \label{wwymuym}
|(U^{\psi} - U^{\psi})\setminus \mathcal{R} | 
& \geq  |U_2^{\psi} - U_1^{\psi}| + |U_1^{\psi} - U_2^{\psi}| +  |(V_2^{\psi} - V_2^{\psi})\setminus \mathcal{R} | \nonumber \\
& = 2|U_2^{\psi} - U_1^{\psi}| +  |(V_2^{\psi} - V_2^{\psi})\setminus \mathcal{R} |.
\end{align}
\par

We now divide our proof into two cases, depending on the dimension of $V_2$.   First, we assume that $\dim(V_2) = u$, or equivalently, $U_1 \subsetneq V_2$. As $\dim(V_2) = u$, there must be at least $u+1$ elements in $V_2$. Moreover, $|V_2| < |U| \leq 2u$, and thus, we can apply $Q(|V_2|,u)$ to get 
\[ |(V_2^{\psi} - V_2^{\psi})\setminus \mathcal{R} |  \geq  2(2u+1-|V_2|) |V_2^{\psi}| - 16 (u^2+ |V_2|^2) . \]
We combine this with $\eqref{wut7}$ and $\eqref{wwymuym}$ to deduce that
\begin{align*}
|(U^{\psi} - U^{\psi})\setminus & \mathcal{R} | 
 \geq  2|U_2^{\psi} - U_1^{\psi}| +  |(V_2^{\psi} - V_2^{\psi})\setminus \mathcal{R} | \nonumber \\
& \geq 2(2u+1-m)|U_2^{\psi}|+  2(2u+1-|V_2|) |V_2^{\psi}| - 16 (u^2+ |V_2|^2)  \\
& \geq 2(2u+1-m)|U_2^{\psi}| + 2(2u + 1 - m)|V_2^{\psi}| - 16 (u^2+ m^2) \\
& = 2(2u+1-m)|U^{\psi}|  - 16 (u^2+ m^2).
\end{align*}
Thus we are done when $\dim(V_2) = u$.
\par

We now focus on the case when $\dim(V_2) = u-1$, that is, when $V_2 = U_1$. We note that $\dim(U_1) + 1 = u \leq q$. Thus if $q \leq m-2 \leq 2(u-1) = 2\dim(U_1)$, we can use the induction hypothesis $Q(|U_1|,u-1)$ to show that
\[  |(U_1^{\psi} - U_1^{\psi})\setminus \mathcal{R} | 
\geq  2(2u + 1 - m)|U_1^{\psi}| - 16( (u-1)^2 + (m-1)^2) . \]
We combine this with $\eqref{wut7}$  and $\eqref{wwymuym}$ to get
 \begin{align*}
  |(U^{\psi} - U^{\psi})\setminus \mathcal{R} |   & \geq  2(2u+1-m)|U^{\psi}| - 16( (u-1)^2 + (m-1)^2) \\
  & \geq 2(2u+1-m)|U^{\psi}| - 16(u^2 + m^2),
 \end{align*}
which is the desired bound.
\par

 We now assume $q= m-1$. This implies that $|U_2| =1$, say, $U_2 = \{y\}$. If $q=m-1 \leq 2u-2= 2\dim(U_1)$, we can use the induction hypothesis $Q(|U_1|,u-1)$ to show that
\[  |(U_1^{\psi} - U_1^{\psi})\setminus \mathcal{R} | 
\geq  2(2u - m)|U_1^{\psi}| - 16( (u-1)^2 + (m-1)^2) . \]
We combine this with $\eqref{fj2}$  and $\eqref{wwymuym}$ to get
 \begin{align*}
  |(U^{\psi} - U^{\psi})\setminus \mathcal{R} |   & \geq  2(2u+1-m)|U^{\psi}|- 2q  - 16( (u-1)^2 + (m-1)^2) \\
  & \geq 2(2u+1-m)|U^{\psi}| - 16(u^2 + m^2),
 \end{align*}
which is the desired conclusion.  Moreover, if $q= m-1 = 2u-1$, then using $\eqref{fj2}$ and $\eqref{wwymuym}$, we see that
 \begin{align*} 
 |(U^{\psi} - U^{\psi})\setminus & \mathcal{R} |  
 \geq |U_2^{\psi} - U_1^{\psi}| + |U_1^{\psi} - U_2^{\psi}|   \geq   2q|U_2^{\psi}|+ 2|U_1^{\psi}| - 2q \\
 & \geq 2(2u+1-m)|U_2^{\psi}| + 2(2u +1-m)|U_1^{\psi}| - 2(m-1) \\
 & = 2(2u+1-m)|U^{\psi}| - 2(m-1),
 \end{align*}
in which case, we are also done. Thus, we finish the inductive step and conclude the proof of Lemma $\ref{highplan}$.
\end{proof}


\section{Proof of Theorem $\ref{diff}$}

Let $A \subseteq \mathbb{R}^d$ be a finite, non-empty set such that $\dim(A) = d$ and $|A| = n$ where $n$ is a large enough natural number. We assume that
\[ |A-A| \leq 8(d-1)|A|.\] 
We use Lemma $\ref{ruzcal}$ with $U= V = A$ to show that
\[ |A+A| \leq |A-A|^3 |A|^{-2} \leq (8d-8)^3 |A|. \]
We now apply Lemma $\ref{fri}$ with $c = (8d-8)^3$ to get parallel lines $l_1, l_2, \dots, l_r$ in $\mathbb{R}^d$, and constants $0 < \sigma \leq 1/2$ and $C_1 > 0$ depending only on $d$ such that 
\begin{equation} \label{ld}
 |A \cap l_1| \geq \dots \geq  |A \cap l_r| \geq  |A \cap l_1|^{1/2} \geq C_1^{-1} n^{\sigma}. 
 \end{equation}
and 
\[ |A\setminus (l_1 \cup l_2 \cup \dots \cup l_r)| < C_1 c n^{1-\sigma}. \]
We write 
\[ S = A \cap (l_1 \cup l_2 \cup \dots \cup l_r), \ \text{and} \ E = A \setminus S. \]
We note that $\eqref{ld}$ implies the bound
\[ |A| \geq |S| = \sum_{i=1}^{r} |A \cap l_i| \geq r C_1^{-1} |A|^{\sigma}.\]
Thus, we have
\begin{equation} \label{ubr}
 r \leq C_1 |A|^{1-\sigma}.
 \end{equation}
\par

We first show that it suffices to prove Theorem $\ref{diff}$ for $S$. Thus, we assume that Theorem $\ref{diff}$ holds for the set $S$. If $\dim(S) = d$, then we have
\begin{align*}
|A-A|  &  \geq |S-S| \geq (2d-2)|S| - O_d(|S|^{1- \delta}) \\
& \geq (2d-2)(|A| - |E|) - O_d(|A|^{1-\delta}) \\
& \geq (2d-2)|A| - O_d(|A|^{1 - \sigma}) + O_d(|A|^{1-\delta}) \\
& = (2d-2)|A| - O_d(|A|^{1- \min{(\sigma, \delta)}}).
\end{align*}
As both $\delta$ and $\sigma$ are strictly positive constants that only depend on $d$, we see that $\min{(\sigma, \delta)}$ is also a strictly positive constant depending only on $d$, and consequently, our claim is proved when $\dim(S) =d$.
\par

Furthermore, if $\dim(S) = d_1 < d$, then there exist linearly independent elements $a_1, \dots, a_{d-d_1} \in E$ such that $\dim(S \cup \{a_1, \dots, a_{d-d_1} \}) = d$. This also implies that $a_1, \dots, a_{d-d_1}$ lie outside the affine span of $S$. Thus, the sets 
\[ S - S, S- a_1, \dots, S-a_{d-d_1}, a_1 - S, \dots, a_{d-d_1} - S \]
are pairwise disjoint. Consequently, we have
\begin{align*}
 |A-A|  &  \geq |S-S| + \sum_{i=1}^{d-d_1} (|S- a_i| + |a_i - S|) \\
 & \geq (2d_1-2)|S|  - O_d(|S|^{1- \delta})  + \sum_{i=1}^{d-d_1} 2|S|\\
 & = (2d - 2) |S| - O_d(|S|^{1-\delta}) \\
 & \geq (2d-2)|A| -  O_d(|A|^{1- \min{(\sigma, \delta)}}).
\end{align*}
\par
As before, we see that $\min{(\sigma, \delta)}$ is a strictly positive constant depending only on $d$, and hence, our claim is proved. Thus, we will now prove a variant of Theorem $\ref{diff}$ for sets contained in a union of parallel lines. 

\begin{Proposition} \label{lines}
Let $d$ be a natural number and let $l_1, l_2, \dots, l_r$ be $r$ parallel lines in $\mathbb{R}^d$. Suppose $A$ is a finite, non-empty, $d$-dimensional subset of $\mathbb{R}^d$ such that 
\[ A \subseteq l_1 \cup l_2 \cup \dots \cup l_r \ \text{and} \ |A \cap l_i| \geq 2d^2  \ (1 \leq i \leq r). \]
Then we have 
\[ |A-A| \geq (2d-2)|A| - K_d r,\]
where $K_d = 1000d^3$.
\end{Proposition}

We remark that Theorem $\ref{diff}$ follows from combining the preceding discussion with $\eqref{ld}$, $\eqref{ubr}$ and Proposition $\ref{lines}$.
\par

We now begin the proof of Proposition $\ref{lines}$. Our strategy will be to follow induction on the dimension $d$ and number of parallel lines $r$ that contain $A$.  Let $P(d,r)$ be the statement of Proposition $\ref{lines}$ for $d$-dimensional sets $A$ which can be covered by $r$ parallel lines. We note that as $\dim(A) = d$, the number of parallel lines $r$ containing $A$ must always be at least $d$. Thus, our base cases will be to prove $P(d,r)$ for all $r \geq d$ such that $d \in \{1,2,3\}$, and $P(d,d)$ for all $d \geq 4$. In our inductive step, for a given $d, r \in \mathbb{N}$ such that $r > d$ and $d \geq 4$, we will show that $P(d,r)$ holds if $P(k,r-1)$ holds for all $k \leq d$. We will, hence conclude that $P(d,r)$ holds for all $r \geq d \geq 3$.
\par

Thus, we begin with the case when $d \in \{1,2,3\}$. In this setting, we can use Lemma $\ref{ruzsa}$ to show that for each finite, non-empty subset $A$ of $\mathbb{R}^d$ such that $\dim(A) = d$, we have
\[|A-A| \geq (d+1)|A| - d(d+1)/2 \geq (2d-2)|A| - d^2. \] 
This finishes our first base case. Our second base case is when $r=d$, where we see that all the sets of the form $(A \cap l_i) - (A \cap l_j)$ are pairwise disjoint whenever $i \neq j$. Consequently, we have
\begin{align*}
|A-A| & \ \geq \ \sum_{i \neq j} |(A \cap l_i) - (A \cap l_j)| \ \geq \ \sum_{i \neq j} (|A \cap l_i| + |A \cap l_j| - 1) \\
        & \ \geq 2(d-1) \sum_{i=1}^{d} |A \cap l_i|  - d^2 \ = \  2(d-1) |A| - d^2 .
\end{align*}
We now move to the inductive step, which will be our primary focus in the next section.


\section{The Inductive Step}

Let $r,d$ be natural numbers such that $r > d \geq 4$. As previously mentioned, we assume that $P(k,r-1)$ holds for all $k \leq d$. Let $H$ be the hyperplane that is orthogonal to $l_1$. For each $1 \leq i \leq r$, we write $x_i$ to be the point where $H$ and $l_i$ intersect, and we let $X = \{ x_1, \dots, x_r \}$. As $\dim(A) = d$, we see that $\dim(X) = d-1.$ Moreover, we denote $\pi$ to be the projection map from $\mathbb{R}^d$ to $H$. For any $Y \subseteq H$, we let $Y^{\pi}$ be a subset of $\mathbb{R}^d$ such that
\[Y^{\pi} = \{ x \in A \ | \ \pi(x) \in Y \}. \]
Thus $Y^{\pi}$ is the pre-image of $Y$ under $\pi$ in $A$. Because we are projecting along the direction of $l_1$ and $|A \cap l_i| \geq 2$ for all $1 \leq i \leq r$, we have $\dim(Y^{\pi}) = \dim(Y) + 1,$ for all $Y \subseteq \pi(A) := \{ \pi(a) \ | \ a \in A \} $. 
\par

We will use ${\|.\|}_d$ to denote the Euclidean norm in $\mathbb{R}^d$. As $H$ is a $(d-1)$-dimensional subspace of $\mathbb{R}^d$, we can find an invertible linear map $\phi$ from $H$ to $\mathbb{R}^{d-1}$. Fixing such a $\phi$, we can induce a norm ${\|.\|}_H$ on $H$ by writing ${\|x\|}_H = {\|\phi(x) \|}_{d-1}$, for all $x \in H$. 
\par

We now consider the convex hull $C$ of $X$. As $\dim(C) = \dim(X) = d-1$, we note that $C$ must have a $(d-2)$-dimensional facet. We denote $H_1$ to be the affine span of this $(d-2)$-dimensional facet and write $l'$ to be the line in $H$ that is orthogonal to $H_1$. 
\par

We cover $X$ with translates of $H_1$ and denote $H_2$ to be the translate of $H_1$ such that $H_2 \cap X \neq \emptyset$ and ${\| (l' \cap H_2) - (l' \cap H_1) \|}_{H}$ is maximised. The existence and uniqueness of such an $H_2$ is confirmed by the fact that $H_1$ is a $(d-2)$-dimensional subspace of $H$, that contains a $(d-2)$-dimensional facet of $C$ and $l'$ is chosen to be orthogonal to $H_1$. 
Thus for all translates $H'$ of $H_1$ such that $H' \cap X \neq \emptyset$ and $H' \neq H_2$, we have
\begin{equation}  \label{far}
{\| (l' \cap H_2) - (l' \cap H_1) \|}_{H} > {\| (l' \cap H') - (l' \cap H_1) \|}_{H}.
\end{equation}
Lastly, for ease of notation, we write $X_i = H_i \cap X$ and $Y_i = X \setminus X_i$ for $i=1,2$.
\par

Our strategy essentially involves analysing how $X_1$, $X_2, Y_1$ and $Y_2$ interact with each other. We divide our proofs into some sub-cases. 

\subsection{}
We first assume that $|X_1^{\pi}| \geq |X_2^{\pi}|$. In this case, we translate our set $A$, and thus $X$, so that $0 \in H_2$. From $\eqref{far}$, we deduce that $X_2 - X_1, X_1 - X_2$ and $Y_2 - Y_2$ are pairwise disjoint. Thus, $X_2^{\pi} - X_1^{\pi}, X_1^{\pi} - X_2^{\pi}$ and $Y_2^{\pi} - Y_2^{\pi}$ are pairwise disjoint. We recall that $\dim(X_1^{\pi}) = \dim(X_1) + 1 = d-1$ and thus, we apply Lemma $\ref{ruzsa}$ to see that
\begin{equation} \label{ind1}
| X_1^{\pi} - X_2^{\pi}| = | X_2^{\pi} - X_1^{\pi}| \geq |X_1^{\pi}| + (d-1)|X_2^{\pi}| - d^2.
\end{equation}

\subsubsection{}
If  $Y_2 = X_1$, we have $\dim(Y_2) = d-2$ and consequently, $\dim(Y_2^{\pi}) = d-1$. Thus, by $P(d-1,r-1)$, we have
\begin{equation} \label{ca1}
 |Y_2^{\pi} - Y_2^{\pi}| \geq 2(d-2)|Y_2^{\pi}| - K_d (r-1).
 \end{equation}
We combine $\eqref{ind1}$ and $\eqref{ca1}$ to get
\begin{align*}
|A-A| & \geq | X_1^{\pi} - X_2^{\pi}| + | X_2^{\pi} - X_1^{\pi}| + |Y_2^{\pi} - Y_2^{\pi}| \\
        & \geq 2(d-2 + 1)|Y_2^{\pi}| + 2(d-1)|X_2^{\pi}| - 2d^2 - K_{d-1}(r-1) \\
        & = (2d-2)|A| - 2d^2 - K_{d-1}(r-1) \\
        & \geq (2d-2)|A| - K_d r.
\end{align*} 

\subsubsection{}
If $Y_2 \neq X_1$, then $\dim(Y_2) = d-1$, which in turn implies that $\dim(Y_2^{\pi}) = d$. Thus, by $P(d,r-1)$, we have
\begin{equation} \label{ca2}
 |Y_2^{\pi} - Y_2^{\pi}| \geq 2(d-1)|Y_2^{\pi}| - K_d (r-1).
 \end{equation}
We now combine $\eqref{ind1}$ and $\eqref{ca2}$ to get
\begin{align*}
|A-A| & \geq  | X_1^{\pi} - X_2^{\pi}| + | X_2^{\pi} - X_1^{\pi}| + |Y_2^{\pi} - Y_2^{\pi}| \\
        & \geq (2d-2)|Y_2^{\pi}| + 2|X_1^{\pi}| + 2(d-1)|X_2^{\pi}| - 2d^2 - K_d (r-1) \\
        & \geq (2d-2)|A| - 2d^2 - K_d(r-1) \\
        & \geq (2d-2)|A| - K_d r.
\end{align*} 

\subsection{}
We now assume that $|X_1^{\pi}| < |X_2^{\pi}|$. In this case, we have
\begin{equation} \label{ind2}
| X_1^{\pi} - X_2^{\pi}| = | X_2^{\pi} - X_1^{\pi}| \geq |X_2^{\pi}| + (d-1)|X_1^{\pi}| - d^2, 
\end{equation} 
instead of $\eqref{ind1}$. Thus, we divide our strategy into different cases depending on $\dim(Y_1)$. 
\par

If $\dim(Y_1) = d-1$, then we can interchange $H_1$ and $H_2$, (and thus, $X_1$ and $X_2$, and $Y_1$ and $Y_2$) and proceed as in $\S 4.1.2$. Furthermore, if $\dim(Y_1) = d-2$, and $Y_1 = X_2$, then we can again interchange $H_1$ and $H_2$, (and consequently, $X_1$ and $X_2$, and $Y_1$ and $Y_2$) and follow $\S4.1.1$. Thus, we can assume that the affine span of $Y_1$ is an affine hyperplane $P$ such that $\dim(P) = k \leq d-2$ and $P$ is not a translate of $H_1$. The rest of our proof will focus on this case.
\par

For ease of notation, we write $\mathcal{P} = \pi^{-1}(P)$ to be the $(k+1)$-dimensional affine subspace in $\mathbb{R}^d$ that is the pre-image of $P$ under $\pi$. As in the previous subcase, we can translate our set $A$ to ensure that $0$ is contained in $H_1$. We decompose $X_1$ into a union of subsets of cosets of $P$, that is, we write
\[ X_1 = Z_0 \cup Z_1 \cup \dots \cup Z_l,\]
where $Z_i$ lies in a distinct coset of $P$ for each $0 \leq i \leq l$, and $Z_0 \subseteq P$. 
 \par

We begin by claiming that the sets
\begin{align}
Z_1 - Y_1, Z_2 - Y_1, \dots, Z_l - Y_1, \label{r1} \\ 
Y_1 - Z_1, Y_1 - Z_2, \dots, Y_1 - Z_l , \label{r2} \\ 
(Y_1 \cup Z_0) - (Y_1 \cup Z_0),   \nonumber \\
 (X_1 - X_1) \setminus (P-P),  \nonumber
\end{align}
are pairwise disjoint. This is true because any two sets in $\eqref{r1}$ are pairwise disjoint as each $Z_i$ lies in a distinct coset of $P$. Moreover, any two sets of the form $Y_1 - Z_i$ and $Z_j - Y_1$ are pairwise disjoint as they lie on opposite sides of the hyperplane $H_1$. Next, we see that any set in $\eqref{r1}$ or $\eqref{r2}$ lies in $H\setminus (H_1 \cup (P-P))$ while $(Y_1 \cup Z_0) - (Y_1 \cup Z_0) \subseteq P-P$ and $((X_1 - X_1)\setminus  (P-P)) \subseteq H_1 \setminus (P-P)$. Thus our claim holds. 
\par

This implies that the sets
\begin{align*}
Z_1^{\pi} - Y_1^{\pi}, Z_2^{\pi} - Y_1^{\pi}, \dots, Z_l^{\pi} - Y_1^{\pi}, \\ 
Y_1^{\pi} - Z_1^{\pi}, Y_1^{\pi} - Z_2^{\pi}, \dots, Y_1^{\pi} - Z_l^{\pi},  \\ 
(Y_1 \cup Z_0)^{\pi} - (Y_1 \cup Z_0)^{\pi},    \\
 (X_1^{\pi} - X_1^{\pi}) \setminus (\mathcal{P}-\mathcal{P}),  
\end{align*}
are pairwise disjoint. Furthermore, we know that 
\[ |Y_1^{\pi}| \geq |X_2^{\pi}| > |X_1^{\pi}| \geq |Z_i^{\pi}|  \ \ (1 \leq i \leq l). \]
Thus, we use Lemma $\ref{ruzsa}$ along with the fact that $\dim(Y_1^{\pi}) = k+1$, to see that
\begin{align*} 
|Z_i^{\pi} - Y_1^{\pi}| + |Y_1^{\pi} - Z_i^{\pi}| 
& \geq 2|Y_1^{\pi}| + 2(k+1)|Z_i^{\pi}| - 2d^2.
\end{align*} 
Summing this for all $1 \leq i \leq l$, we get
\begin{equation} \label{p1}
\sum_{i=1}^{l} (|Z_i^{\pi} - Y_1^{\pi}| + |Y_1^{\pi} - Z_i^{\pi}| ) \geq 2l |Y_1^{\pi}| + 2(k+1)\sum_{i=1}^{l} |Z_i^{\pi}|-2ld^2. 
\end{equation} 
As $k \leq d-2$, we can use $P(k+1,r-1)$ to deduce that
\begin{align} 
|(Y_1 \cup Z_0)^{\pi} - (Y_1 \cup Z_0)^{\pi}| 
& \geq 2k |(Y_1 \cup Z_0)^{\pi}| - K_{k+1}(r-1) \nonumber \\
& = 2k|Y_1^{\pi}| + 2k|Z_0^{\pi}| - K_{k+1}(r-1)      \label{p2}
\end{align} 
Combining $\eqref{p1}$ and $\eqref{p2}$ with the fact that $A-A$ contains the pairwise disjoint sets $Z_1^{\pi} - Y_1^{\pi}, \dots,  (X_1^{\pi} - X_1^{\pi}) \setminus (\mathcal{P} - \mathcal{P})$, we get
\begin{equation} \label{res}
|A-A|  \geq  2(l+k)|Y_1^{\pi}| + 2k|X_1^{\pi}|  +  |(X_1^{\pi} - X_1^{\pi}) \setminus (\mathcal{P} - \mathcal{P})| + 2\sum_{i=1}^{l}|Z_i^{\pi}| - C, 
\end{equation}
where $C = K_{k+1}(r-1) + 2ld^2$. 
\par

Note that $l \geq d-1-k$, otherwise we could construct an affine subspace of dimension at most $d-2$ that contains $X$. This would contradict the fact that $\dim(X) = d-1$. Moreover, as $|Y_1^{\pi}| > |X_1^{\pi}|$, we see that
\begin{equation} \label{pff}
2(l+k)|Y_1^{\pi}| + 2k|X_1^{\pi}| \geq 2(d-1)|Y_1^{\pi}| + 2(l+2k -(d-1))|X_1^{\pi}|. 
\end{equation}
\par

We translate our set $A$ to ensure that $P$ is a proper subspace and we split our proof into two cases depending on whether $Z_0$ is empty or not. If $Z_0 = \emptyset$, then we have $\sum_{i=1}^{l} |Z_i^{\pi}| = |X_1^{\pi}|$. Thus, combining $\eqref{res}$ and $\eqref{pff}$, we get
\begin{equation} \label{was1}
 |A-A| \geq 2(d-1)|Y_1^{\pi}| + 2(l+2k -d+2)|X_1^{\pi}|+  |(X_1^{\pi} - X_1^{\pi}) \setminus (\mathcal{P} - \mathcal{P})| - C.
 \end{equation}
In this case, if $l+2k -d+2 \geq d-1$, then we are done since $|A| = |X_1^{\pi}| + |Y_1^{\pi}|$. Thus, we can assume that $l \leq 2(d-k-2)$.
\par

We now consider the projection maps $\psi : \mathbb{R}^d \to \mathbb{R}^d/(\mathcal{P}-\mathcal{P})$, and $\phi : H \to H/(P-P)$. We write
\[ U = \psi(X_1^{\pi}) = \{ \psi(x) \ | \ x \in X_1^{\pi} \}, \ \text{and} \ U' = \phi(X_1) = \{ \phi(x)  \ | \ x \in X_1 \}. \]
Since $\pi(\mathbb{R}^d) = H$ and $\pi(\mathcal{P} - \mathcal{P}) = P-P$ and $\pi(X_1^{\pi}) = X_1$, we see that $U$ and $U'$ are isomorphic sets, whence, $\dim(U) = \dim(U')$ and $|U| = |U'|$. We use $u$ to denote $\dim(U')$. Since $Z_0 = \emptyset$, we note that $|U| = l$. Moreover, since $Z_0 = \emptyset$ and $\dim(U) = u$, we see that $U$ is contained in a subspace of dimension $u+1$. This further implies that $u+1 + k \geq d-1$, as otherwise we could construct an affine subspace of dimension at most $d-2$ that contains $X$. 
\par

This can be summarised as follows 
\[ 2u \geq 2(d-2-k) \geq l = |U| \geq u+1. \]
Thus, we can apply Lemma $\ref{highplan}$ to get
\[ |(X_1^{\pi} - X_1^{\pi}) \setminus (\mathcal{P} - \mathcal{P})| \geq 2(2(d-2-k) +1 -l)|X_1^{\pi}| - 80d^2. \]
Combining this with $\eqref{was1}$, we find that
\[ |A-A| \geq 2(d-1)|Y_1^{\pi}| + 2(d-1)|X_1^{\pi}| - C - 80d^2 \geq 2(d-1) |A| - K_{d}r, \]
in which case, we are done.
\par

When $Z_0 \neq \emptyset$, we proceed similarly. In particular, we use $\eqref{res}$ and $\eqref{pff}$ to get
\begin{equation} \label{was2}
|A-A|  \geq  2(d-1)|Y_1^{\pi}| + 2(l+2k -(d-1))|X_1^{\pi}| +  |(X_1^{\pi} - X_1^{\pi}) \setminus (\mathcal{P} - \mathcal{P})|  - C.
\end{equation}
As before, if $l + 2k - (d-1) \geq d-1$, we would be done, whence we can assume that $l +1 \leq 2(d-1-k)$. Next, we consider the projection maps $\psi$ and $\phi$, and the projected sets $U$ and $U'$ as defined earlier. Since $Z_0 \neq \emptyset$, we have $|U| = l+1$. Furthermore, as $Z_0 \neq \emptyset$ and $\dim(U) = u$, we see that $U$ is contained in a subspace of dimension $u$. Consequently, we have $u + k \geq d-1$, otherwise we could construct an affine subspace of dimension at most $d-2$ that contains $X$. Summarising this as before, we have
\[ 2u \geq 2(d-1-k) \geq l+1 = |U| \geq u+1.\]
Thus, we can apply lemma $\ref{highplan}$ to deduce that 
\[ |(X_1^{\pi} - X_1^{\pi}) \setminus (\mathcal{P} - \mathcal{P})| \geq 2(2(d-1-k) +1 -(l+1))|X_1^{\pi}| - 80d^2. \]
Putting this together with $\eqref{was2}$, we get 
\[ |A-A| \geq  2(d-1)|Y_1^{\pi}| + 2(d-1)|X_1^{\pi}| - C - 80d^2 \geq 2(d-1) |A| - K_{d}r,\]
which is the desired bound. Thus, we finish the inductive step and conclude the proof of Proposition $\ref{lines}$.

\bibliographystyle{amsbracket}
\providecommand{\bysame}{\leavevmode\hbox to3em{\hrulefill}\thinspace}

\end{document}